\documentclass[english, 10pt]{amsart}

\usepackage{babel}
\usepackage{amstext}
\usepackage{amsmath}
\usepackage{amssymb}
\usepackage{amsfonts}
\usepackage{latexsym}
\usepackage{ifthen}
\usepackage[all,2cell,dvips]{xy} \UseAllTwocells \SilentMatrices
\xyoption{all}
\usepackage{verbatim}

\newcommand{\rk}{{\rm rk}}

\newtheorem{lemma1}{}[section]

\newenvironment{lemma}{\begin{lemma1}{\bf Lemma.}}{\end{lemma1}}
\newenvironment{example}{\begin{lemma1}{\bf Example.}\rm}{\end{lemma1}}

\newenvironment{theorem}{\begin{lemma1}{\bf Theorem.}}{\end{lemma1}}
\newenvironment{proposition}{\begin{lemma1}{\bf Proposition.}}{\end{lemma1}}
\newenvironment{corollary}{\begin{lemma1}{\bf Corollary.}}{\end{lemma1}}
\newenvironment{remark}{\begin{lemma1}{\bf Remark.}\rm}{\end{lemma1}}
\newenvironment{definition}{\begin{lemma1}{\bf Definition.}}{\end{lemma1}}

\newenvironment{remark*}{{\bf Remark.}}{}
\newenvironment{example*}{{\bf Example.}}{}

\newcommand{\Z}{\ensuremath{\mathbb{Z}}}
\newcommand{\C}{\ensuremath{\mathbb{C}}}

\newcommand{\PP}{\ensuremath{\mathbb{P}}}

\newcommand{\holom}[3]{\ensuremath{#1:#2  \rightarrow #3}}
\newcommand{\fibre}[2]{\ensuremath{#1^{-1} (#2)}}

\makeatletter
\ifnum\@ptsize=0  \fi \ifnum\@ptsize=2  \fi 

\newcommand\sF{{\mathcal F}}

\newcommand\sI{{\mathcal I}}

\newcommand\sO{{\mathcal O}}

\DeclareMathOperator*{\sing}{sing}
\DeclareMathOperator*{\Nm}{Nm}

\newcommand\pic{\ensuremath{\mbox{Pic} \ }}
\newcommand\picd{\ensuremath{\mbox{Pic}}}

\newcommand{\ppav}{principally polarised abelian variety}
\newcommand{\ppavs}{principally polarised abelian varieties}

\newcommand{\upC}{\widetilde{C}}

\newcommand{\pull}{(\pi_1^*, \pi_2^*)}

 \usepackage{palatino}

\title{Geometry of Brill-Noether loci on Prym varieties} 
\date{March 5, 2011}

\author{Andreas H\"oring}

\thanks{Partially supported by the A.N.R. project ``CLASS''}

\address{Andreas H\"oring, Universit\'{e} Pierre et Marie Curie and Albert-Ludwig Universit\"at Freiburg}
\curraddr{Mathematisches Institut, Albert-Ludwigs-Universit\"at
  Freiburg, Eckerstra{\ss}e 1, 79104 Freiburg im Breisgau, Germany}
\email{hoering@math.jussieu.fr}

\begin{document}

\begin{abstract}
Given the Prym variety of an \'etale double cover
one can define analogues of the classical Brill-Noether loci on Jacobians of curves.
Recent work by Lahoz and Naranjo shows that the Brill-Noether locus $V^2$ completely
determines the covering. In this paper we describe the singular locus and
the irreducible components of $V^2$.
\end{abstract}

\maketitle

\section{Introduction}

Given a smooth curve $X$ it is well-known that the
Brill-Noether loci $W^r_d X$ contain a lot of interesting information 
about the curve $X$ and its polarised Jacobian $(JX, \Theta_X)$.
Given a smooth curve $C$ and an \'etale double cover \holom{\pi}{\tilde{C}}{C},
one can define analogously Brill-Noether loci $V^r$ for the Prym variety $(P, \Theta)$ (cf. Section \ref{sectionnotation}).
Several fundamental results on these loci are known for quite a while: 
the expected dimension is $g(C)-1-{r+1 \choose 2}$,
they are non-empty if the expected dimension is non-negative \cite[Thm.1.4]{Ber87}
and connected if the expected dimension is positive \cite[Ex.6.2]{Deb00}.
If $C$ is general in the moduli space of curves, all the Brill-Noether loci are smooth and
have the expected dimension \cite[Thm.1.11]{Wel85}.
While the Brill-Noether locus $V^1 \subset P^+$ is the canonically defined theta-divisor and has received the
attention of many authors, the study of the higher Brill-Noether loci and the information
they contain about the \'etale cover \holom{\pi}{\tilde{C}}{C} is a very recent subject:
Casalaina-Martin, Lahoz and Viviani \cite{CMLV08} have shown that $V^2$ is set-theoretically the theta-dual (cf. Definition \ref{definitionthetadual})
of the Abel-Prym curve. In their recent preprint Lahoz and Naranjo \cite{LN10} refine this statement
and prove a Torelli theorem: the Brill-Noether locus $V^2$ determines the covering $\upC \rightarrow C$.
This justifies a more detailed study of the geometry of $V^2$. 
Our first result is the

\begin{theorem} \label{maintheorem} 
Let $C$ be a smooth curve of genus $g(C) \geq 6$, and let \holom{\pi}{\tilde{C}}{C} be
an \'etale double cover such that the Prym variety $(P, \Theta)$ is an irreducible \ppav\footnote{The condition
on the irreducibility is always satisfied unless $C$ is hyperelliptic, but $\upC$ is not. In this case  $(P, \Theta)$ is isomorphic
to a product of Jacobians \cite{Mum74b}.}. 

a) Suppose that $C$ is hyperelliptic. Then $V^2$ is irreducible of dimension $g(C)-3$.

b) Suppose that $C$ is not hyperelliptic. Then $V^2$ is a reduced Cohen-Macaulay scheme of dimension $g(C)-4$.
If the singular locus $V^2_{\sing}$ has an irreducible component of dimension at least $g(C)-5$, then $C$ is a plane quintic, trigonal or bielliptic.
\end{theorem}

While in the hyperelliptic case (cf. Proposition \ref{propositionhyperelliptic}) the statement is a straightforward extension of \cite{CMLV08}, the non-hyperelliptic case is based on the following observation: 
if the singular locus of $V^2$ is large, the singularities are exceptional in the sense of \cite{Bea82b}.
This provides a link with certain Brill-Noether loci on $JC$.

An immediate consequence of the theorem is that
$V^2$ is irreducible unless $C$ is a plane quintic, trigonal or bielliptic (Corollary \ref{corollaryreducible}). 
The case of trigonal curves is very simple: $(P, \Theta)$ is isomorphic to a Jacobian $JX$ and $V^2$
splits into two copies of $W^0_{g(C)-4} X$. 
For a plane quintic $V^2$ is reducible if and only if $(P, \Theta)$
is isomorphic to the intermediate Jacobian of a cubic threefold;
in this case $V^2$ splits into two copies of the Fano surface $F$.
Note that the Fano surface $F$ and the Brill-Noether loci $W^0_d X$ are conjectured
to be the only subvarieties of \ppavs \ having the minimal cohomology class $[\frac{\Theta^{k}}{k!}]$ \cite{Deb95}. 
By \cite{dCP95} the cohomology class of $V^2$ is $[2 \frac{\Theta^{g(C)-4}}{(g(C)-4)!}]$,
a reducible $V^2$ provides thus an important test for this conjecture.
Our second result is the

\begin{theorem} \label{theoremirreduciblecomponents}
Let $C$ be a smooth non-hyperelliptic curve of genus $g(C) \geq 6$, and let \holom{\pi}{\tilde{C}}{C} be
an \'etale double cover. Denote by $(P, \Theta)$ the polarised Prym variety.
The Brill-Noether locus $V^2$ is reducible if and only if one of the following holds:

a) $C$ is trigonal;

b) $C$ is a plane quintic and $(P, \Theta)$  an intermediate Jacobian of a cubic threefold;

c) $C$ is bielliptic and the covering \holom{\pi}{\tilde{C}}{C} belongs to the family
$\mathcal R_{\mathcal B_{g(C), g(C_1)}}$ with $g(C_1) \geq 2$ (cf. Remark \ref{remarkbielliptic}).
Then $V^2$ has two or three irreducible components, but none of them has minimal cohomology class.
\end{theorem}

It is known \cite{Sho81} that if $C$ is bielliptic of genus $g(C) \geq 8$, 
the Prym variety is not a Jacobian of a curve.
Moreover we know by \cite{Deb88} that these Prym varieties form 
$\lfloor \frac{g(C)-1}{2} \rfloor$ distinct subvarieties of $\mathcal A_{g(C)-1}$.
For exactly one of these families
the general member has the property that the cohomology class of {\em any} subvariety is a multiple
of the minimal class $\frac{\Theta^k}{k!}$. 
The proof of Theorem \ref{theoremirreduciblecomponents} shows that the Brill-Noether locus $V^2$ is irreducible if and only if the Prym variety belongs to this family!
This is the first evidence for Debarre's conjecture that is 
not derived from low-dimensional cases or 
considerations on Jacobians and intermediate Jacobians \cite{Ran80, Deb95, a5}.

{\bf Acknowledgements.}
The work of O. Debarre, M. Lahoz and J.-C. Naranjo plays an important r\^ole in this paper.
I want to thank them for patiently answering my numerous questions.

\section{Notation} \label{sectionnotation}

While most of our arguments are valid for an arbitrary algebraically closed field
of characteristic $\neq 2$, we will work over $\C$: this is necessary to apply
\cite{ACGH85} and \cite{Deb00} which are crucial for Theorem \ref{maintheorem} and its consequences.
For standard definitions in algebraic geometry we refer to \cite{Har77}, for Brill-Noether theory
to \cite{ACGH85}. 

Given a smooth curve $C$ we denote by $\pic C$ its Picard scheme and by
$$
\pic C = \bigcup_{d \in \Z} \picd^d C
$$ 
the decomposition in its irreducible components. 
We will identify the Jacobian $JC$ and the degree $0$ component $\picd^0 C$
of the Picard scheme.
In order to simplify the notation we denote by $L \in \pic C$ the point
corresponding to a given line bundle $L$ on $C$. Somewhat abusively we will say that a line bundle is effective
if it has a global section.

If \holom{\varphi}{X}{Y} is a finite cover between smooth curves and $D$ a divisor on $X$, we denote by
$\Nm \varphi(D)$ its norm. In the same way \holom{\Nm \varphi}{\pic X}{\pic Y} denotes the norm map.
If $\sF$ is a coherent sheaf on $X$ (in general $\sF$ will be the locally free sheaf corresponding to some divisor),
we denote by $\varphi_* \sF$ the push-forward as a sheaf.

Let $C$ be a smooth curve of genus $g(C)$ and \holom{\pi}{\tilde{C}}{C} 
an \'etale double cover. We have 
$\fibre{(\Nm \pi)}{K_C} = P^+ \cup P^-$,
where $P^- \simeq P^+ \simeq P$ are defined by
$$
P^-:=\{ L\in \fibre{(\Nm \pi)}{K_C} \: | \: \dim |L|\equiv 0 \mod 2\},
$$
$$
P^+:=\{ L\in \fibre{(\Nm \pi)}{K_C} \: | \: \dim |L| \equiv 1 \mod 2\}.
$$
For $r \geq 0$ we denote by 
$$
W^r_{2g(C)-2} \upC := \{ L\in \picd^{2g(C)-2} \upC \: | \: \dim |L| \geq r \}.
$$
The Brill-Noether loci of the Prym variety \cite{Wel85} are defined as the scheme-theoretical 
intersections
$$
V^r := W^r_{2g(C)-2} \upC \cap P^- \qquad \mbox {if $r$ is even}
$$
and
$$
V^r := W^r_{2g(C)-2} \upC \cap P^+ \qquad \mbox {if $r$ is odd.}
$$

The notion of theta-dual was introduced by Pareschi and Popa in their work on Fourier-Mukai transforms (cf. \cite{PP08b} 
for a survey).

\begin{definition} \label{definitionthetadual}
Let $(A, \Theta)$ be a \ppav, and let $X \subset A$ be any closed subset.
The theta-dual $T(X)$ of $X$ is the closed subset defined by
\[
T(X) := \{
\alpha \in A \ | \
h^0(A, \sI_X(\Theta) \otimes \alpha) \neq 0 
\}. 
\] 
\end{definition}

Note that $T(X)$ has a natural scheme structure \cite{PP08b}, 
set-theoretically it is defined by $X-T(X) \subset \Theta$.

\section{The singular locus of $V^2$}

In the whole section we denote by $C$ a smooth non-hyperelliptic curve of genus $g(C)$ 
and by \holom{\pi}{\tilde{C}}{C} an \'etale double cover. 
The following lemma will be used many times:

\begin{lemma} \label{lemmatangentspace}
Let $L \in V^r$ be a line bundle such that $\dim |L|=r$.
If the Zariski tangent space $T_L V^r$ satisfies
$$
\dim T_{L} V^r > g(C)-2r,
$$
there exists a line bundle $M$ on $C$ such that $\dim |M| \geq 1$ and an effective line bundle $F$ on $\tilde{C}$
such that
$$
L \simeq \pi^* M \otimes F.
$$
\end{lemma}

\begin{remark}
For $r=1$ the scheme $V^1 = W^1_{2g(C)-2} \upC \cap P^+$ identifies to 
the canonical polarisation $\Theta$. The theta-divisor
has dimension $g(C)-2$, so the condition
$$
\dim T_L V^1 > g(C)-2,
$$
is equivalent to $V^1$ being singular in $L$. 
Thus for $r=1$ we obtain the well-known statement that if a point $L \in \Theta$ with $\dim |L|=1$ is 
in $\Theta_{\sing}$, the singularity is exceptional in the sense of Beauville \cite{Bea82b}.
\end{remark}

\begin{proof}
We consider the Prym-Petri map introduced by Welters \cite[1.8]{Wel85}
$$
\holom{\beta}{\wedge^2 H^0(\upC,L)}{H^0(\upC,K_{\upC})^-},
$$
where $H^0(\upC,K_{\upC})^-$ identifies to the tangent space of the Prym variety, in particular it has
dimension $g(C)-1$.
By \cite[Prop. 1.9]{Wel85} the Zariski tangent space of $V^r$
at the point $L$ equals the orthogonal of the image of $\beta$.
Thus if $\dim T_{L} V^r > g(C)-2r$, then $\rk \beta<2r-1$.
Since $\wedge^2 H^0(\upC, L)$ has dimension $\frac{r(r+1)}{2}$, this is equivalent to
$$
(*) \qquad \dim \ker \beta > \frac{r(r+1)}{2} - 2r-1.
$$
The locus of decomposable 2-forms in $\wedge^2 H^0(\upC, L)$
is the affine cone over the Pl\"ucker embedding of
$G(2, H^0(\upC, L))$ in $\PP(\wedge^2 H^0(\upC, L))$,
so it has dimension $2r-1$. 
Thus by $(*)$ there is a  non-zero decomposable vector $s_i\wedge s_j$ in $\ker
\beta$. 
This means that $s_i \sigma^*s_j- s_j \sigma^*s_i=0$, 
thus $\frac{s_j}{s_i}$ defines a rational function
$h$ on $C$. We conclude by taking $M=\sO_C((h)_0)$ 
and $F$ the maximal common divisor between $(s_i)_0$ and $(s_j)_0$.
By construction $F$ is effective and $\dim |M| \geq 1$.
\end{proof}

By \cite[Thm.2.2]{CMLV08},  \cite[Lemma 2.1]{IP01} 
every irreducible component of the Brill-Noether locus $V^2$ has 
dimension at most $g(C)-4$ if $C$ is not hyperelliptic. 
The following estimate is a generalisation of their statement to arbitrary $r$.

\begin{lemma} \label{lemmadimensionestimate}
We have
$$ 
\dim V^r \leq g(C)-2-r \qquad \forall \ r \geq 2.
$$
\end{lemma}

\begin{proof} 
Denote by $|K_C| \subset C^{(2g(C)-2)}$ the set of effective canonical divisors
and by  \holom{\Nm \pi}{\tilde{C}^{(2g(C)-2)}}{C^{(2g(C)-2)}} the norm map.
Since the canonical linear system $|K_C|$ defines an embedding we know by
\cite[\S 2,Cor.]{Bea82b} that \fibre{\Nm \pi}{|K_C|} has exactly two irreducible components $\Lambda_0$ and $\Lambda_1$, both are normal varieties of dimension $g(C)-1$.
Let
$$
\holom{i}{\tilde{C}^{(2g(C)-2)}}{\picd^{2g(C)-2} \tilde{C}},  \ D \mapsto \sO_{\upC}(D)
$$
be the Abel-Jacobi map, then (up to renumbering)
$$
\varphi(\Lambda_0) = P^- \qquad \mbox{and} \qquad 
\varphi(\Lambda_1) = \Theta  \subset P^+.
$$
Recall that for all $L \in \pic \upC$ we have a set-theoretic equality $\fibre{i}{L} = |L|$.
In particular we see that 
\[
(*) \qquad \dim \fibre{i}{V^r} \geq \dim V^r + r
\]
for every $r \geq 0$. 

Suppose now that $r$ is even (the odd case is analogous and left to the reader).
For a general point $L \in P^-$ one has $\dim |L|=0$.
Thus for $r \geq 2$ one has  
\[
\fibre{i}{V^r} \subsetneq \Lambda_1,
\]
hence $\fibre{i}{V^r}$ has dimension at most $g(C)-2$. Conclude by $(*)$.
\end{proof}

\begin{remark} \label{remarkdimensionestimatehyperelliptic}
In the preceding proof the hypothesis $C$ not hyperelliptic was only used to show
that $\Lambda_0$ and $\Lambda_1$ are irreducible. Since the inequality $(*)$ is valid without this 
property we obtain that
$$
\dim V^r \leq g(C)-1-r \qquad \forall \ r \geq 2.
$$
We will see in Section \ref{subsectionhyperelliptic} that this estimate is optimal.
\end{remark}

We can now use Marten's theorem to give an estimate of the dimension of the
singular locus $V^2_{\sing}$.

\begin{proposition}\label{propositiondimsingularities}
Suppose that $g(C) \geq 6$ and $V^2_{\sing}$ has an irreducible component $S$ of dimension at least $g(C)-5$.
Then there exists a $d \in \{ 3, 4 \}$ such that 
$$
\dim W^1_d C = d-3
$$
and $W \subset W^1_d C$ an irreducible component of maximal dimension such that for every $M \in W$ one has
$$
\dim |K_C \otimes M^{\otimes -2}| = g-d-2.
$$
For every $L$ in $S$ one has
$$
L \simeq \pi^* M \otimes F
$$
for some $M \in W$ and some effective line bundle $F$ on $\upC$.
In particular $S$ has dimension equal to $g(C)-5$.
\end{proposition}

\begin{proof}
Let $L \in S$ be a generic point, then by Lemma \ref{lemmadimensionestimate}
we have $\dim |L|=2$.
Since $V^2$ is singular in $L$ we have
$$
\dim T_{L} V^2 > g(C)-4.
$$
Thus by Lemma \ref{lemmatangentspace} there
exists a line bundle $M \in W^1_d C$ for some $d \leq g(C)-1$ and an effective line bundle $F$ on $\tilde{C}$
such that
$$
L \simeq \pi^* M \otimes F.
$$
The family of such pairs $(M,F)$ is a finite cover of the set of pairs $(M, B)$
where $M \in W^1_d C$ for some $d \leq g(C)-1$ and
$B$ is an effective divisor of degree $2g(C)-2-2d \geq 0$ on $C$ 
such that $B \in |K_C\otimes M^{\otimes -2}|$.

By hypothesis the parameter space $T$ of the pairs $(M,B)$ has dimension at least $g(C)-5$.
Note that if $\deg M=g(C)-1$ then $K_C\otimes M^{\otimes -2} \simeq \sO_C$. 
Thus $M$ is a theta-characteristic and the space of pairs $(M,B)$ is finite, a
contradiction to $g(C)-5>0$. Since $C$ is not hyperelliptic we get $3 \leq \deg M<g(C)-1$.
Moreover one has by Clifford's theorem 
\[
(*) \qquad \dim |H^0(C, K_C\otimes M^{\otimes -2})| \leq g(C)-1-d-1.
\]
Thus the variety $W$ parametrising the line bundles $M$ has dimension at least $d-3$.
By construction $W \subset W^1_d$ and by Marten's theorem \cite[IV, Thm.5.1]{ACGH85}
\[
(**) \qquad \dim W \leq  \dim W^1_d C \leq d-3.
\]
Thus $T$ and $S$ have dimension at most $g(C)-5$. Since by hypothesis $S$ has dimension at least $g(C)-5$, we see that
$(*)$ and $(**)$ are equalities, at least for $M \in W$ generic. By upper semicontinuity and 
Clifford's theorem we get equality for every $M \in W$. 

The last remaining point is to show that this situation can only occur
for $d \in \{ 3, 4 \}$:
by what precedes we have a finite map 
$$
W \rightarrow W_{2g(C)-2-2d}^{g-d-2} C, \ M \mapsto K_C \otimes M^{\otimes -2}.
$$
If $2g(C)-2-2d \leq g(C)-1$ we know by Marten's theorem that $\dim W_{2g(C)-2-2d}^{g(C)-d-2} C \leq 1$.
Since $\dim W=d-3$ we get $d \leq 4$.

If $2g(C)-2-2d \geq g(C)$ we use the isomorphism 
$$
W_{2g(C)-2-2d}^{g(C)-d-2} C \rightarrow W^{d-1}_{2d} C, \ K_C \otimes M^{\otimes -2} \mapsto M^{\otimes 2}
$$
and Marten's theorem to see that $\dim W_{2g(C)-2-2d}^{g(C)-d-2} C \leq 1$, so we get again $d \leq 4$.
\end{proof}

\begin{proof}[Proof of Theorem \ref{maintheorem}]
The hyperelliptic case is settled in Proposition \ref{propositionhyperelliptic}, so suppose $C$ is not hyperelliptic.

By \cite[Ex.6.2.1)]{Deb00} the Brill-Noether-locus $V^2$ is a determinantal variety.
Since for $C$ not hyperelliptic it has the expected dimension, it is Cohen-Macaulay. 
Since $\dim V^2_{\sing} \leq g(C)-5$ by Proposition \ref{propositiondimsingularities} all the irreducible 
components of $V^2$ are generically reduced. A generically reduced Cohen-Macaulay scheme is reduced.

If $\dim V^2_{\sing} \geq g(C)-5$ then  $\dim W^1_d C=d-3$ for $d=3$ or $4$ by Proposition \ref{propositiondimsingularities}.
Thus the second statement follows from Mumford's refinement of Marten's theorem \cite[IV., Thm.5.2]{ACGH85}.
\end{proof}

\begin{remark*}
Based on completely different methods Lahoz and Naranjo \cite{LN10} have shown that $V^2$ is
reduced and Cohen-Macaulay.
\end{remark*}

\begin{corollary} \label{corollaryreducible} 
Let $C$ be a smooth non-hyperelliptic curve of genus $g(C) \geq 6$, 
and let \holom{\pi}{\tilde{C}}{C} be an \'etale double cover.
If $V^2$ is reducible, then $C$ is a plane quintic, trigonal or bielliptic.
\end{corollary}
  
\begin{proof}  
By a theorem of Debarre \cite[Ex.6.2.1)]{Deb00} the locus $V^2$ is $g(C)-5$-connected, i.e.
if $V^2$ is not irreducible, there exist two irreducible components $Z_1, Z_2 \subset V^2$ such that
$Z_1 \cap Z_2$ has dimension at least $g(C)-5$ in one point (cf. ibid, p.287). Thus if $V^2$ is reducible,
its singular locus has dimension at least $g(C)-5$. Conclude with Theorem \ref{maintheorem}.
\end{proof}

\section{Examples}

\subsection{Hyperelliptic curves} \label{subsectionhyperelliptic}

Let $C$ be a smooth hyperelliptic curve of genus $g(C)$ 
and \holom{\pi}{\tilde{C}}{C} an \'etale double cover
such that the Prym variety $(P, \Theta)$ is an irreducible \ppav, i.e.
$\upC$ is also a hyperelliptic curve. Let $\sigma: \upC \rightarrow \upC$ be the involution induced by $\pi$.

Recall \cite[Ch.12]{BL04} that in this case
the Abel-Prym map
$$
\alpha: \upC \rightarrow P, \ p \mapsto \sigma(p)-p
$$ 
is two-to-one onto its image $C'$ which is a smooth curve
and the Prym variety $(P, \Theta)$ is isomorphic to $(J(C'), \Theta_{C'})$. 

In \cite[Lemma 2.1]{CMLV08} the authors show that for $C$ not hyperelliptic, $V^2$ is a translate of
the theta-dual of the Abel-Prym embedded curve $\upC \subset P$.
In fact their argument works also for $C$ hyperelliptic if one replaces $\upC \subset P$ by 
$\alpha(\upC)=C' \subset P$. Thus we have:

\begin{lemma}
The Brill-Noether locus $V^2$ is a translate of the theta-dual $T(C')$.
\end{lemma}

Since the Prym variety $(P, \Theta)$ is isomorphic to $(J(C'), \Theta_{C'})$, 
the theta-dual of $C'$ is a translate of $W^0_{g(C)-3} C'$. In particular $V^2$ is irreducible of
dimension $g(C)-3$.

\begin{proposition} \label{propositionhyperelliptic} 
Let $C$ be a smooth hyperelliptic curve of genus $g(C) \geq 6$, 
and let \holom{\pi}{\tilde{C}}{C} be an \'etale double cover 
such that the Prym variety $(P, \Theta)$ is an irreducible \ppav.
Then $V^2$  is irreducible of dimension $g(C)-3$,
set-theoretically it is a translate $W^0_{g(C)-3} C'$.

If $L \in V^2$ is any point, then
$$
L \simeq \pi^* H \otimes F
$$
where $H$ the unique $g^1_2$ on $C$ and $F$ is an effective line bundle on $\upC$. 
\end{proposition}

\begin{proof}
By Remark \ref{remarkdimensionestimatehyperelliptic} we have a proper inclusion $V^4 \subsetneq V^2$,
so a general $L \in V^2$ satisfies $\dim |L|=1$.
By Lemma \ref{lemmatangentspace} there
exists a line bundle $M \in W^1_d C$ for some $d \leq g(C)-1$ and an effective line bundle $F$ on $\tilde{C}$
such that
$$
L \simeq \pi^* M \otimes F.
$$
We can now argue as in the proof of Proposition \ref{propositiondimsingularities} to obtain the statement, the
only thing we have to observe is that the inequality 
\[
(*) \qquad \dim |H^0(C, K_C\otimes M^{\otimes -2})| \leq g(C)-1-d-1.
\]
is also valid on a hyperelliptic curve unless $M$ is a multiple of the $g^1_2$.
\end{proof}

\subsection{Plane quintics} \label{subsectionquintic}

Let $C \subset \PP^2$ be a smooth plane quintic and \holom{\pi}{\tilde{C}}{C} an \'etale double cover.
We denote by $H$ the restriction of the hyperplane divisor to $C$, and by $\eta \in \picd^0 C$ the two-torsion line bundle 
inducing $\pi$.
 Let $\sigma: \upC \rightarrow \upC$ be the involution induced by $\pi$.

\begin{example} \label{exampleintermediate}
Suppose that $h^0(C, \sO_C(H) \otimes \eta)$ is odd, i.e. the Prym variety  
is isomorphic to the intermediate Jacobian of a cubic threefold \cite{CG72}. 
The Fano variety $F$ parametrising lines on the threefold is a smooth surface
that has a natural embedding in the intermediate Jacobian.
The surface $F \subset P$ has minimal cohomology class $[\frac{\Theta^3}{3!}]$ (ibid). 
Moreover it follows from \cite{a4} and \cite{PP08} that the theta-dual satisfies $T(F)=F$. 
It is well-known that $\upC \subset F$ (up to translation), so 
$$
F=T(F) \subset V^2 =T(\upC).
$$ 
Since the condition $\dim |L| \geq 2$ is invariant under isomorphism, the 
Brill-Noether locus $V^2$ is stable under the map $x \mapsto -x$.
Thus $F \subset V^2$ implies that $-F \subset V^2$. Since the
cohomology class of $V^2$ is $[2 \frac{\Theta^3}{3!}]$ we see that up to translation
$V^2$ is a union of  $F$ and $-F$. 
In particular $V^2$ is reducible and its singular locus is the intersection of the two irreducible components.
Since $V^2$ is Cohen-Macaulay, the singular locus has pure dimension one.
\end{example}

We will now prove the converse of the example:

\begin{proposition} \label{propositionplanequintics}
The Brill-Noether locus $V^2$ is reducible if and only if 
$h^0(C, \sO_C(H) \otimes \eta)$ is odd, i.e. if and only if the Prym variety
is isomorphic to the intermediate Jacobian of a cubic threefold.
 
In this case the singular locus $V^2_{\sing}$ is a translate of $\upC$.
\end{proposition}

\begin{proof} 
Suppose that $V^2_{sing}$ has a component $S$ of dimension one.
Since $C$ is not trigonal we know by Proposition \ref{propositiondimsingularities}
that $S$ corresponds to a one-dimensional component $W \subset W^1_4 C$ such that  
for every $[M] \in W$ 
$$
|K_C \otimes M^{\otimes -2}| \neq \emptyset.
$$
By adjunction $K_C \simeq \sO_C(2H)$ and by \cite[\S 2, (iii)]{Bea82a} we have
$M \simeq \sO_C(H-p)$ where $p \in C$ is a point. Thus $K_C \otimes M^{\otimes -2} \simeq \sO_C(2p)$ 
and a general point $L \in S$ is of the form
$$
L \simeq \pi^* \sO_C(H-p) \otimes \sO_{\upC}(q_1+q_2)
$$
where $q_1, q_2$ are points in $\upC$. Since $\Nm \pi L \simeq \sO_C(2H)$ and $C$
is not hyperelliptic we obtain that $q_i \in \fibre{\pi}{p}$. Thus we can write
$$
L \simeq \pi^* H \qquad \mbox{or} \qquad L \simeq  \pi^* \sO_C(H) \otimes \sO_{\upC}(q-\sigma(q)) 
\ \mbox{for some} \ q \in \upC. 
$$
Since $L$ varies in a one-dimensional family we can exclude the first case.
By Mumford's description of Prym varieties whose theta-divisor has a singular locus of dimension $g(C)-5$,
we know (cf. \cite[p.347, l. -4]{Mum74b}) that $h^0(C, \sO_C(H) \otimes \eta)$ is even if and only if
$h^0(\upC, \pi^* \sO_C(H) \otimes \sO_{\upC}(q-\sigma(q)))$ is even. 
Since $V^2 \subset P^-$ this shows the statement.

The description of the general points $L \in S$ shows that $V^2_{\sing}$ has a unique one-dimensional component 
and that it is the translate by $\pi^* \sO_C(H)$ of the Abel-Prym embedded $\upC \subset P$.
\end{proof}

\subsection{Trigonal curves} \label{subsectiontrigonal}

Let $C$ be a trigonal curve of genus $g(C) \geq 6$.
Let \holom{\pi}{\tilde{C}}{C} an \'etale double cover, and $(P, \Theta)$ the corresponding Prym variety.
By Recillas' theorem \cite{Rec74} the Prym variety is isomorphic as a \ppav \ to
the polarised Jacobian $(J X, \Theta_X)$ of a tetragonal curve $X$ of genus $g(C)-1$. 
By Recillas's construction \cite[Ch.12.7]{BL04} we also know how to recover the double cover \holom{\pi}{\tilde{C}}{C}
from the curve $X$: let \holom{s}{X^{(2)} \times X^{(2)}}{X^{(4)}} be the sum map, then 
$$
\tilde C \simeq p_1(\fibre{s}{\PP^1}),
$$  
where $\PP^1 \subset X^{(4)}$ is the linear system giving the tetragonal structure and $p_1$ the projection
on the first factor. In particular we see that
$$
\tilde C \subset X^{(2)} \simeq W^0_2 X.
$$
Thus (up to choosing an isomorphism $(P, \Theta) \simeq (J X, \Theta_X)$ and appropriate translates) one has
\[
T(W^0_2 X) \subset T(\upC) \simeq V^2.
\]
By \cite[Ex.4.5]{PP08} the theta-dual of $W^0_2 X$ is $-W^0_{g(C)-4} X$.
As in the case of the intermediate Jacobian \ref{exampleintermediate} we see that up to translation
\[
V^2 = -W^0_{g(C)-4} X \cup W^0_{g(C)-4} X,
\]
and the singular locus of $V^2$ is the union of $\pm (W^0_{g(C)-4} X)_{\sing}$, which has
dimension at most $g(C)-6$ and
the intersection of the two irreducible components, which has dimension $g(C)-5$.

\section{Prym varieties of bielliptic curves I}

\subsection{Special subvarieties}  \label{subsectionspecial}
We recall some well-known facts about special subvarieties which we will use in the next section.

Let \holom{\varphi}{X}{Y} be a double cover (which may be \'etale or ramified) of smooth curves.
We suppose that $g(Y)$ is at 
least one, and denote by \holom{\Nm \varphi}{\pic X}{\pic Y} the norm morphism. 
Let $M$ be a globally generated line bundle of degree $d \geq 2$ on $Y$.
Denote by $\PP^r \subset Y^{(d)}$ where $r:=\dim |M|$ the set of effective divisors in the linear system $|M|$.
If  \holom{\Nm \varphi}{X^{(d)}}{Y^{(d)}} is the norm map,
then $\Lambda:=\fibre{\Nm \varphi}{\PP^r}$ is a reduced Cohen-Macaulay scheme of pure dimension $r$ and the map $\Lambda \rightarrow |M|$
is \'etale of degree $2^d$ over the locus of smooth divisors in $|M|$ which do not meet the
branch locus of $\varphi$. 

If $\varphi$ is \'etale, $\Lambda$ has exactly two connected components $\Lambda_0$ and $\Lambda_1$ 
\cite{Wel81}.
If $\varphi$ is ramified, the scheme $\Lambda$ is connected \cite[Prop.14.1]{Nar92}.
Let 
$$
\holom{i_Y}{Y^{(d)}}{J Y}, \ D \mapsto \sO_Y(D)
$$ 
and 
$$
\holom{i_X}{X^{(d)}}{J X},  \ D \mapsto \sO_{X}(D)
$$
be the Abel-Jacobi maps, then we have a commutative diagram
\[
\xymatrix{
\Lambda \ar @{^{(}->}[r] \ar[d] & X^{(d)} \ar[r]^{i_X} \ar[d]_{\Nm \varphi}
& \picd^d X \ar[d]^{\Nm \varphi}
\\
\PP^r \ar @{^{(}->}[r]  & Y^{(d)} \ar[r]_{i_Y} & \picd^d Y
}
\]
The fibre of $i_X (X^{(d)}) \rightarrow i_Y (Y^{(d)})$ over the point $M$
(and thus the intersection of  $i_X (X^{(d)})$ with $\fibre{\Nm \varphi}{M}$) 
is equal (at least set-theoretically) to $i_X(\Lambda)$.

Fix now a connected component $S \subset \Lambda$. Then we call $V:=i_X(S)$
{\em a special subvariety\footnote{In general it is not true that $S$ is irreducible, in particular
the special subvariety may not be a variety.
Note also that in general it should be clear which covering we consider, otherwise
we say that $V$ is a $\varphi$-special subvariety associated to $M$.}} associated to $M$. 
Obviously one has 
\begin{equation} \label{equationdimensionspecial}
\dim V = r- \dim |\sO_X(D)|
\end{equation}
where $D \in S$ is a general point.

The following technical definition will be very important in the next section:

\begin{definition} \label{definitionsimple}
Let \holom{\varphi}{X}{Y} be a double cover of smooth curves. An effective divisor $D \subset X$ is not simple if there exists a point $y \in Y$ such that $\varphi^* y \subset D$. It is simple if this is not the case.
\end{definition}

Note that if an effective divisor $D \subset X$ is not simple, then $\Nm \varphi(D)$ is not reduced. Hence if $Y$ is an elliptic curve and $M$ a line bundle of degree $d \geq 2$ on $Y$, then a general divisor $D \in X^{(d)}$ such that $\Nm \varphi(D) \in |M|$
is simple: the linear system $|M|$ is base-point free, so a general element is reduced. 

\begin{lemma} \label{lemmaspecialoverelliptic}
Let \holom{\varphi}{X}{Y} be a ramified double cover of smooth curves such that $Y$ is an elliptic curve. 
Denote by $\delta_\varphi$ the line bundle of degree $g(X)-1$ defining the cyclic cover $\varphi$.
Let $M \not\simeq \delta_\varphi$ be a line bundle of degree $2 \leq d \leq g(X)-1$ on $Y$.

a) Then $\Lambda$ is smooth and irreducible.

b) A general divisor $D \in \Lambda$ is simple and satisfies $\dim |\sO_X(D)|=0$.

In particular there exists a unique special subvariety associated to $M$,
it is irreducible of dimension $d-1$.
\end{lemma}

\begin{proof}
We start by showing the statement $b)$: by what precedes $D$ is simple, so by \cite[p.338]{Mum71} we have an exact sequence
\[
0 \rightarrow \sO_Y \rightarrow \varphi_* \sO_X(D) \rightarrow 
\sO_Y(\Nm \varphi(D)) \otimes \delta_\varphi^* \rightarrow 0.
\]
Since $\deg D \leq \deg \delta_\varphi$ and $\sO_Y(\Nm \varphi(D)) \simeq M \not\simeq \delta_\varphi$ we have 
$h^0(Y,  \sO_Y(\Nm \varphi(D)) \otimes \delta_\varphi^*)=0$.
Hence we have $1=h^0(Y, \sO_Y) = h^0(Y, \varphi_* \sO_X(D))$.

For the proof of $a)$ note first that since $\Lambda$ is connected, it is sufficient
to show the smoothness.
Let $D \in \Lambda$ be any divisor then we have a unique decomposition
$$
D = \varphi^* A + R + B,
$$
where $A$ is an effective divisor on $Y$, the divisor $R$ is effective 
with support contained in the ramification locus of $\varphi$ and $B$ is effective, simple
and has support disjoint from the ramification locus of $\varphi$.
Since $Y$ is an elliptic curve, we have
$$
h^0(Y, M \otimes \sO_Y(-A-\Nm \varphi(R))) 
= h^0(Y, M) - \deg (A+\Nm \varphi(R))
$$ 
unless $\deg M = \deg (A+\varphi_* R)$ and $M \otimes \sO_Y(-A-\Nm \varphi(R))$ is not trivial.
Since $\deg M=\deg D$ this last case could only happen when 
$A=0$ and $B=0$, so we have $D=R$. Yet by construction $M \simeq \sO_Y(\Nm \varphi(D)) = \sO_Y(\Nm \varphi(R))$,
so $M \otimes \sO_Y(-A-\Nm \varphi(R))$ is trivial.
By \cite[Prop.14.3]{Nar92} this shows the smoothness of $\Lambda$, the statement on the dimension
follows by $b)$ and Equation \eqref{equationdimensionspecial}.
\end{proof}

\subsection{The irreducible components of $V^2$}
\label{subsectionz2z2}

In this section $C$ will be a smooth curve of genus $g(C) \geq 6$ that is bielliptic, i.e.
we have a double cover \holom{p}{C}{E} onto an elliptic curve $E$.
As usual \holom{\pi}{\upC}{C} will be an \'etale double cover.
In this section we suppose that the covering \holom{p \circ \pi}{\upC}{E} is Galois. 
In this case one sees easily that the Galois group is $\Z_2 \times \Z_2$.

Using the Galois action on $\upC$ we get a commutative diagram\footnote{Our presentation follows \cite[Ch.5]{Deb88} to which we
refer for details.}
\[
\xymatrix{
& \tilde{C}  \ar[ld]_\pi \ar[d]^{\pi_1} \ar[rd]^{\pi_2} &
\\
C \ar[rd]_p & C_1  \ar[d]^{p_1} & C_2 \ar[ld]^{p_2}
\\
& E &
}
\]
It is straightforward to see that
\[
g(C_1)+g(C_2)=g(C)+1,
\]
and we will assume without loss of generality that $1 \leq g(C_1) \leq g(C_2) \leq g(C)$.
We denote by $\Delta$ the branch locus of $p$ and by $\delta$ the line bundle inducing the cyclic cover $p$.
Then we have $2 \delta \simeq \Delta$, moreover by the Hurwitz formula $\deg K_C = \deg \Delta$, hence
\[
\deg \delta = g(C)-1.
\]
Analogously the cyclic covers $p_1$ and $p_2$ are given by line bundles $\delta_1$ und $\delta_2$ such that
$\deg \delta_1 = g(C_1)-1$ and $\deg \delta_2 = g(C_2)-1$.

For any $a \in \Z$ we define closed subsets $Z_a \subset \pic C_1 \times \pic C_2$ by 
\[
\left\{
(L_1, L_2) \ | \ L_1 \in W^0_{g(C_1)-1+a} C_1, \ L_2  \in W^0_{g(C_2)-1-a} C_2, \ \Nm p_1(L_1) \otimes \Nm p_2(L_2) \simeq \delta
\right\}.
\]
We note that the sets $Z_a$ are empty unless $1-g(C_1) \leq a \leq g(C_2)-1$. 
Pulling back to $\upC$ we obtain natural maps
$$
\pull: Z_a \rightarrow \pic \upC, \ (L_1, L_2) \mapsto \pi_1^* L_1 \otimes \pi_2^* L_2
$$
and by \cite[p.230]{Deb88} the image  $\pull(Z_a)$ is in $P^-$ if and only if $a$ is odd.
Moreover we can argue as in \cite[Prop.5.2.1]{Deb88} to see that
\begin{equation} \label{v2inclusion1}
V^2 \subset  \pull(\bigcup_{a \ \mbox{\tiny odd}}Z_a).
\end{equation}

\begin{lemma} \label{lemmaza}
For $a$ odd the sets $Z_a$ are empty or one has
\begin{equation} \label{dimensionza}
\dim Z_a = g(C)-1-a.
\end{equation}
Moreover $Z_a$ is irreducible unless $g(C_1)=1$ and $a \geq g(C_2)-2$.
\end{lemma}

\begin{proof}
{\em 1st case. $g(C_1)>1$.} 
We prove the statement for positive $a$, for $a$ negative the argument is analogous.
The projection on the second factor gives a surjective map $Z_a \rightarrow W^0_{g(C_2)-1-a} C_2$,
the fibres of this map being parametrized by effective line bundles $L_1$ with fixed norm.
Since $a \geq 1$ the line bundles $L_1$ are of degree at least $g(C_1)$, so they are automatically effective. 
Thus the fibres identify to fibres of the norm map $\Nm p_1: \pic C_1 \rightarrow \pic E$.
Since the double covering $p_1$ is ramified, the $\Nm p_1$-fibres are irreducible of dimension $g(C_1)-1$, 
so $Z_a$ is irreducible of the expected dimension.

{\em 2nd case. $g(C_1)=1$.} The sets $Z_a$ are empty for $a$ negative, so suppose $a$ positive.
Arguing as in the first case we obtain the statement on the dimension.
In order to see that $Z_a$ is irreducible for $a \leq g(C_2)-3$ we consider the surjective map induced
by the projection on the first factor $Z_a \rightarrow \picd^{g(C_1)-1+a} C_1$. The fibre over a line bundle $L_1$
is the union of the $p_2$-special subvarieties associated to $\delta \otimes \Nm p_1 L_1^*$. Since
$2 \leq \deg \delta \otimes \Nm p_1(L_1^*) \leq g(C_2)-2$ we know by Lemma \ref{lemmaspecialoverelliptic} 
that the unique special subvariety is irreducible, so the fibres are irreducible.
\end{proof}

Since all the irreducible components of $V^2$ have dimension $g(C)-4$ 
if follows from \eqref{v2inclusion1} and \eqref{dimensionza} that
\begin{equation} \label{v2inclusion}
V^2 \subset \pull(\bigcup_{a \ \mbox{\tiny odd}, |a| \leq 3} Z_a).
\end{equation}
If $(L_1,L_2) \in Z_{\pm 3}$ then by Riemann-Roch $\dim |L_1| \geq 2$ (resp. $\dim |L_2| \geq 2$), so we have
$$
\pull(Z_{\pm 3}) \subset V^2.
$$ 
For the sets $Z_{\pm 1}$ this can't be true, since Equation \eqref{dimensionza} shows that 
they have dimension $g(C)-2$. We introduce the following smaller loci:
\[
W_1 := 
\left\{
(L_1, L_2) \in Z_1  \ | \ L_1 \in W^1_{g(C_1)} C_1
\right\},
\]
and 
\[
W_{-1} := 
\left\{
(L_1, L_2) \in Z_{-1}  \ | \ L_2 \in W^1_{g(C_2)} C_2
\right\}.
\]
We note that if $g(C_1)=1$, then $W_1=\emptyset$: there is no $g^1_1$ on a non-rational curve.
Since $\dim W^1_{g(C_1)} C_1=g(C_1)-2$ (resp. $\dim W^1_{g(C_2)} C_1=g(C_2)-2$)  
one deduces easily from the proof of Lemma \ref{lemmaspecialoverelliptic} that the sets $W_{\pm 1}$ are empty or
irreducible of dimension $g(C)-4$.

By the same lemma we see that for fixed $L_1$ (resp. $L_2$) and general $L_2$ (resp. $L_1$) such that
$(L_1, L_2) \in W_1$  (resp. $(L_1, L_2) \in W_{-1}$), the linear system $|L_1|$ (resp. $|L_2|$) contains a unique effective
divisor and this divisor is simple.

Note that if $(L_1,L_2) \in W_{\pm 1}$, then $\dim |\pull(L_1,L_2)| \geq 1$. Since these sets map into the
component $P^-$ we obtain
$$
\pull(W_{\pm 1}) \subset V^2.
$$
\begin{proposition} \label{propositioncomponents}
We have
$$
V^2 = \pull(Z_{-3} \cup W_{-1} \cup W_{1} \cup Z_{3}).
$$
\end{proposition}

The proof needs some technical preparation:

\begin{definition} \label{definitionsimplelinebundle}
Let \holom{\varphi}{X}{Y} be a double cover of smooth curves.
Let $L$ be a line bundle on $X$ such that $\dim |L| \geq 1$. The line bundle $L$ is simple if every divisor in $D \in |L|$
is simple in the sense of Definition \ref{definitionsimple}.  
\end{definition}

\begin{lemma} \label{lemmatechnical} \cite[Cor.5.2.8]{Deb88}
In our situation let $L_1 \in \pic C_1$ and $L_2 \in \pic C_2$ be effective line bundles such
that $L \simeq \pi_1^* L_1 \otimes \pi_2^* L_2$.
If $L_1$ is $p_1$-simple, then
\[
h^0(\upC, L) \leq 2 h^0(C_2, L_2)+g(C_2)-1-\deg L_2.
\]
Analogously if $L_2$ is  $p_2$-simple, then
\[
h^0(\upC, L) \leq 2 h^0(C_1, L_1)+g(C_1)-1-\deg L_1.
\]
\end{lemma}

\begin{proof}[Proof of Proposition \ref{propositioncomponents}]
Let $L \in V^2$ be an arbitrary line bundle. By the inclusion \eqref{v2inclusion} we are left to show that if 
$L \in \pull(Z_{\pm 1})$ and $L \not\in \pull(Z_{-3} \cup Z_{3})$, then $L \in \pull(W_{\pm 1})$.
We will suppose that $L \in \pull(Z_{1})$, the other case is analogous and left to the reader.
Since $L \in \pull(Z_1)$, we can write
\[
L \simeq \pi_1^* L_1 \otimes \pi_2^*L_2
\]
with $L_1$ effective of degree $g(C_1)$    
and $L_2$ effective of degree $g(C_1)-2$. 
If $L_2$ is not simple, then $L$ is in $\pull(Z_3)$ which we excluded.
Hence $L_2$ is simple, so by Lemma \ref{lemmatechnical} we obtain
\[
3 \leq h^0(\upC, L) \leq 2 h^0(C_1, L_1)+g(C_1)-1-g(C_1).
\]
Thus one has $\dim |L_1| \geq 1$ and $L \in \pull(W_{1})$.
\end{proof}

\begin{corollary} \label{corollarygone}
If $g(C_1)=1$, then 
$$
V^2 = \pull(Z_{3}).
$$
In particular $V^2$ is irreducible.
\end{corollary}

\begin{proof}
Since the sets $Z_{-3}, W_{-1}$ and $W_{1}$ are empty for $g(C_1)=1$, the first statement is immediate
from Proposition \ref{propositioncomponents}. Since $g(C_1)=1$ implies that $g(C_2)=g(C)$ and $g(C) \geq 6$ 
by hypothesis we know by Lemma \ref{lemmaza} that $Z_3$ is irreducible.
\end{proof}

We will now focus on the case $g(C_1) \geq 2$.
Proposition \ref{propositioncomponents} reduces the study of $V^2$ to understanding the sets
$W_{\pm 1}, Z_{\pm 3}$ and their images in $P^-$. We start with an observation:

\begin{lemma} \label{lemmaidentify}
For $g(C_1) \geq 2$ we have
\[
\pull(W_1) = \pull(W_{-1}).
\]
\end{lemma}

\begin{proof}
We claim that the following holds:
let $L_1 \in W^1_{g(C_1)} C_1$ be a general point. Then $L_1$ is not simple, 
and there exists a point $x \in E$ such that  
$$
L_1 \simeq p_1^*\sO_E(x) \otimes \sO_{C_1}(D_1)
$$ 
with $D_1$ an effective divisor such that $\sO_E(\Nm p_1(D_1)+x) \simeq \delta_1$. 
Assume this for the time being, let us show how to conclude: let $L \in  \pull(W_1)$ 
be a general point. Then we have $L \simeq \pi_1^* L_1 \otimes \pi_2^* L_2$ with $L_1 \in W^1_{g(C_1)} C_1$ a general point
and $L_2$ a $p_2$-simple line bundle. Thus by the claim we can write
$$
L \simeq \pi_1^*  \sO_{C_1}(D_1) \otimes \pi_2^* (L_2 \otimes p_2^* \sO_E(x)).
$$
Since $\sO_E(\Nm p_1(D_1)+x) \simeq \delta_1$ and  $\delta \simeq \delta_1 \otimes \delta_2$ a short computation 
shows that $\Nm p_2(L_2) \otimes \sO_E(x) \simeq \delta_2$.
Moreover $L_2$ is $p_2$-simple, so by \cite[Prop.5.2.7]{Deb88} we obtain that $\dim |L_2 \otimes p_2^* \sO_E(x)| \geq 1$.
Thus $L$ is in $\pull(W_{-1})$.
This shows one inclusion, the proof of the other inclusion is analogous.

{\em Proof of the claim.} Set
\[
S := 
\{
(x, D_1) \in E\ \times C_1^{(g(C_1)-2)} \ | \ x+\Nm p_1(D_1) \in |\delta_1|
\}\footnote{For $g(C_1)=2$, the symmetric product $C_1^{(g(C_1)-2)}$ is a point: it corresponds to the zero divisor on $C_1$.}.
\]
Note that the projection $p_2: S \rightarrow C_1^{(g(C_1)-2)}$ on the second factor is an isomorphism,
so $S$ is not uniruled.  
For $(x, D_1) \in S$ general the divisor $D_1$ is $p_1$-simple by Lemma \ref{lemmaspecialoverelliptic}, 
so by \cite[p.338]{Mum71} we have an exact sequence
\[
0 \rightarrow \sO_E(x) \rightarrow (p_1)_* \sO_{C_1}(p_1^* x+D_1) \rightarrow \sO_E(x+\Nm p_1(D_1)) \otimes \delta_1^* \rightarrow 0.
\]
By construction we have $\sO_E(x+\Nm p_1(D_1)) \otimes \delta_1^* \simeq \sO_E$. Thus $H^1(E, \sO_E(x))=0$ implies
that $h^0(C_1, \sO_{C_1}(p_1^* x+D_1))=2$.
Hence the image of
\[
\holom{\tau}{S}{\pic C_1}, \  (x,D_1) \mapsto \sO_{C_1}(p_1^* x+D_1)
\]
is contained in $W^1_{g(C_1)} C_1$. 
Since $S$ is not uniruled the general fibre
of $S \rightarrow \tau(S)$ has dimension zero.
By Riemann-Roch the residual map 
$W^1_{g(C_1)} C_1 \rightarrow W^0_{g(C_1)-2} C_1$ is an isomorphism, 
so $W^1_{g(C_1)} C_1$ is irreducible of dimension $g(C_1)-2$.
Thus $\tau$ is surjective.
\end{proof}

Suppose that $g(C_1) \geq 2$.
Let $(J C_1, \Theta_{C_1})$ and $(J C_2, \Theta_{C_2})$ be the Jacobians of the curves $C_1$ and $C_2$
with their natural principal polarisations. Since $p_1$ and $p_2$ are ramified, the pull-backs
\holom{\pi_1^*}{JE}{JC_1} and \holom{\pi_2^*}{JE}{JC_2} are injective 
and the restricted polarisations  $B_1:=\Theta_{C_1}|_{JE}$ and $B_2:=\Theta_{C_2}|_{JE}$
are of type $(2)$ \cite[Ch.3]{Mum71}.
We define
\[
P_1 := \ker(\holom{\Nm p_1}{J C_1}{J E}), \ P_2:= \ker(\holom{\Nm p_2}{J C_2}{J E}).
\]
We set $A_1:=\Theta_{C_1}|_{P_1}$ and $A_2:=\Theta_{C_2}|_{P_2}$, then the polarisations $A_1$ and $A_2$ are of type
$(1, \ldots, 1, 2)$ \cite[Cor.12.1.5]{BL04}. 

If  $\holom{p_j^* \times i_{P_j}}{JE \times P_j}{JC_j}$
denotes the natural isogeny, then $(p_j^* \times i_{P_j})^* \Theta_{C_j} \equiv B_j \boxtimes A_j$.
Thus if \holom{\alpha_j}{JC_j}{JE \times \widehat{P_j}} is the dual map, one has \cite[Prop.14.4.4]{BL04} 
\begin{equation} \label{decompositionthetac}
\Theta_{C_j}^{\otimes 2} \equiv \alpha_j^* (\widehat{B_j} \boxtimes \widehat{A_j}),
\end{equation}
where $\widehat{B_j}$ and $\widehat{A_j}$ are the dual polarisations. We note that $\widehat{A_j}$ has type 
$(1, 2, \ldots, 2)$.

By \cite[Prop.5.5.1]{Deb88} the pull-back maps $P_1$ and $P_2$ into the Prym variety $P$ and we obtain an isogeny
\holom{\pull|_{P_1 \times P_2}}{P_1 \times P_2}{P} such that
\[
\pull|_{P_1 \times P_2}^* \Theta \equiv A_1 \boxtimes A_2.
\]
In particular if \holom{g}{P}{\widehat{P_1 \times P_2}} denotes the dual map, 
one has 
\begin{equation} \label{decompositiontheta}
\Theta^{\otimes 2} \equiv g^* (\widehat{A_1} \boxtimes \widehat{A_2}).
\end{equation}

\begin{proposition} \label{propositiongtwo}
If $g(C_1) \geq 3$, the cohomology classes of $\pull(Z_{-3})$, $\pull(W_{1})$, $\pull(Z_{3})$
are not minimal. Moreover their cohomology classes are distinct, so they are distinct irreducible
components of $V^2$.

If $g(C_1)=2$ the same holds for $\pull(W_{1})$ and $\pull(Z_{3})$.
\end{proposition}

\begin{proof}
In order to simplify the notation we denote the pull-back of the polarisations 
$\widehat{A_1}$ and $\widehat{A_2}$ to $\widehat{P_1 \times P_2}$ by the same letter.

We start by observing that is sufficient to show that $[\pull(Z_{-3})]$ (resp. $[\pull(Z_{-3})]$)
is a non-negative multiple of $g^* \widehat{A_1}^3$ ($g^* \widehat{A_2}^3$).  
Indeed once we have shown this property we can use that 
$$
[\pull(Z_{-3})] + [\pull(Z_{3})] + [\pull(W_1)] = [V^2]=\frac{\Theta^3}{3!}
$$
and Formula \eqref{decompositiontheta} to compute that
\[
 [\pull(W_1)] = \frac{1}{3! 2^3} 
 \left[
 (1-a_1) g^* \widehat{A_1}^3 + 3 g^* \widehat{A_1}^2 \widehat{A_2} + 3 g^* \widehat{A_1} \widehat{A_2}^2 + (1-a_2) g^* \widehat{A_2}^3
 \right],
\]
where $a_1, a_2 \geq 0$ correspond to the cohomology class of $Z_{\pm 3}$.
It is clear that none of these classes is (a multiple of) a minimal cohomology class.  
If $g(C_1) \geq 4$ all the classes are non-zero and distinct, 
so the images of $Z_{\pm 3}$ and $W_1$ are distinct irreducible components of $V^2$.
If $2 \leq g(C_1) \leq 3$ the set $Z_{-3}$ is empty (and the corresponding class zero), so
we obtain only two irreducible components.

{\em Computation of the cohomology class of $\pull(Z_{\pm 3})$.}
We will prove the claim for $Z_3$, the proof for $Z_{-3}$ is analogous.
We have a commutative diagram
\[
\xymatrix{
P \ar @{^{(}->}[r]^{i_P} & J \upC  \ar[r]^{\simeq} & \widehat{J \upC} \ar[r]^{\widehat{i_P}} \ar[d]_{\widehat{\pull}} & \widehat{P} \simeq P \ar[d]_{g}
\\
& JC_1 \times JC_2 \ar[r]^{\simeq}  \ar[u]^{\pull} & \widehat{JC_1 \times JC_2} \ar[r]^q & \widehat{P_1 \times P_2} 
},
\]
so if $X \subset JC_1 \times JC_2$ is a subvariety such that $\pull(X) \subset P$, its cohomology class is determined (up to a multiple) by the class of $q(X)$ in $\widehat{P_1 \times P_2}$.

We choose a translate of $Z_3$ that is in $JC_1 \times JC_2$ and denote it by the same letter.
We want to understand the geometry of $q(Z_3)$. 
Since the norm maps $\Nm p_j$ are dual to the pull-backs $p_j^*$ \cite[Ch.1]{Mum71}, the map $q$ fits into an exact sequence of abelian varieties
\begin{equation} \label{mapq}
0 \rightarrow JE \times JE \stackrel{p_1^* \times p_2^*}{\rightarrow} JC_1 \times JC_2 
\stackrel{q}{\rightarrow} \widehat{P_1 \times P_2}  \rightarrow 0
\end{equation}
Recall from the proof of Lemma \ref{lemmaza} that $Z_3$ is a fibre space over $W^0_{g(C_1)-4}$ such that for given $L_2 \in W^0_{g(C_2)-4}$, the fibre identifies to the fibre of $\Nm p_1:\picd^{g(C_1)+2} C_1 \rightarrow \picd^{g(C_1)+2} E$ over $\delta \otimes \Nm p_2(L_2^*)$. Thus $Z_3$ identifies to a fibre product 
$$
\picd^{g(C_1)+2} C_1 \times_{J E} W^0_{g(C_2)-4}.
$$ 
Together with the exact sequence \eqref{mapq} this shows that 
\[
q(Z_3) = \widehat{P_1} \times q_2(W^0_{g(C_2)-4}),
\]
where \holom{q_2}{JC_2}{P_2} is the restriction of $q$ to $JC_2$.

Thus we are left to compute the cohomology class of $q_2(W^0_{g(C_2)-4})$: 
note first that $q_2$ is the composition of the isogeny
\holom{\alpha_2}{JC_2}{JE \times \widehat{P_2}} with the projection on $\widehat{P_2}$. 
Since the polarization $\widehat{B_2}$ is numerically equivalent to a multiple of 
$e \times \widehat{P_2} \subset JE \times \widehat{P_2}$ and the cohomology class of $W^0_{g(C_2)-4}$
is $\frac{\Theta_{C_2}^4}{4!}$, one deduces from 
Equation \eqref{decompositionthetac} that the cohomology class of $q_2(W^0_{g(C_2)-4})$
is a multiple of $\widehat{A_2}^3$.
\end{proof}

\begin{remark} \label{remarkclasses}
With some more effort one can show the following statement:
If $g(C_1) \geq 2$, then the following equalities in $H^6(P, \Z)$ hold:
\begin{eqnarray}
\label{one}
[\pull(Z_{-3})] &=& \frac{1}{4!}  g^* p_{\widehat{P_1}}^* \widehat{A_1}^3,
\\
\label{two}
[\pull(Z_{3})] &=& \frac{1}{4!}  g^* p_{\widehat{P_2}}^*  \widehat{A_2}^3,
\\
\label{three}
[\pull(W_1)] &=& \frac{1}{8}  g^* (p_{\widehat{P_1}}^* \widehat{A_1}^2 p_{\widehat{P_2}}^* \widehat{A_2}
+p_{\widehat{P_1}}^* \widehat{A_1} p_{\widehat{P_2}}^* \widehat{A_2}^2),
\end{eqnarray}

The polarisation $\widehat{A_j}$ being of type $(1, 2, \ldots, 2)$ it follows from \cite[Thm.4.10.4]{BL04}
that $\frac{1}{4!} \widehat{A_j}^3$ is a ``minimal'' cohomology class for $(P_j, \widehat{A_j})$, i.e.
it is in $H^6(P_j, \Z)$ and not divisible.  
\end{remark}

\begin{remark} \label{remarkbielliptic}
Let $\mathcal R_{g(C)}$ be the moduli space of pairs $(C, \pi)$ where $C$ is a smooth projective curves of genus $g(C)$
and $\holom{\pi}{\upC}{C}$ an \'etale double cover. 
We denote by
\[
\holom{\mbox{Pr}}{\mathcal R_{g(C)}}{\mathcal A_{g(C)-1}}
\]
the Prym map associating to $(C, \pi)$ the principally polarised Prym variety $(P, \Theta)$.

Let $\mathcal B_{g(C)}$ be the moduli space of bielliptic curves of genus $g(C) \geq 6$,
and let $\mathcal R_{\mathcal B_{g(C)}} \subset \mathcal R_{g(C)}$ be the moduli space of \'etale double covers 
over them. 
Let $\mathcal R_{\mathcal B_{g(C), g(C_1)}}$ be those \'etale double covers
such that $\upC \rightarrow C \rightarrow E$ has Galois group $\Z_2 \times \Z_2$
and the curve $C_1$ has genus $g(C_1)$.

By \cite[Thm.4.1.i)]{Deb88} the closure of 
$\mbox{Pr}(\mathcal R_{\mathcal B_{g(C), 1}})$ in $\mathcal A_{g(C)-1}$ contains the locus 
of Jacobians of hyperelliptic curves of genus $g(C)-1$.  
A general hyperelliptic Jacobian has the property that 
the cohomology class of every subvariety is an integral multiple of the 
minimal class \cite{Bis97}. 
Hence the same property holds for a general element in $\mbox{Pr}(\mathcal R_{\mathcal B_{g(C), 1}})$ 
Thus if $V^2$ was reducible, the irreducible components would have minimal cohomology class.
\end{remark}

\section{Prym varieties of bielliptic curves II}

\subsection{Tetragonal construction and $V^2$}
\label{subsectiontetragonal}

We denote by $C$ an irreducible nodal curve of arithmetic genus
$p_a(C) \geq 6$, and by \holom{\pi}{\upC}{C} a Beauville admissible cover. 
By \cite{Bea77} the corresponding Prym variety $(P, \Theta)$
is a \ppav.
We will suppose that $C$ is a tetragonal curve, i.e. there exists
a finite morphism \holom{f}{C}{\PP^1} of degree four,
but not hyperelliptic, trigonal or a plane quintic. 
We set $H := f^* \sO_{\PP^1}(1)$.
By Donagi's tetragonal construction \cite{Don81}, \cite[Ch.12.8]{BL04}
the corresponding special subvarieties give Beauville admissible  covers
$\upC' \rightarrow C'$ and $\upC'' \rightarrow C''$
such that  $C'$ and $C''$ are tetragonal and the Prym varieties are 
isomorphic to $(P, \Theta)$.

Consider now the residual line bundle $K_C \otimes H^*$: by Riemann-Roch
the linear series $|K_C \otimes H^*|$ is a $g^{p_a(C)-4}_{2p_a(C)-6}$
to which we can apply the construction of special subvarieties (cf. Section \ref{subsectionspecial}).
If $S \subset \Lambda$ is a connected component, then by \cite[Thm.1, Rque.4]{Bea82b} 
the cohomology class of $V:=i_{\upC}(S)$ is $[2 \frac{\Theta^3}{3!}]$.
Denote by $(P^+, \Theta^+)$ the canonically polarised Prym variety, i.e.
\[
\Theta^+ = \{ L\in \fibre{(\Nm \pi)}{K_C} \: | \: |L| \neq \emptyset, \dim |L|\equiv 0 \mod 2 \}.
\]
Up to exchanging $\upC'$ and $\upC''$ we can suppose that the image of the natural map
\[
V \times \upC' \rightarrow J \upC
\]
is contained in $P^+$. By construction the image is then contained in $\Theta^+$, hence
a translate of $-V$ is contained in the theta-dual $T(\upC')$. 
Since $T(\upC')$ equals the Brill-Noether locus $(V^2)'$ of the covering $\upC' \rightarrow C'$,
it has cohomology class $[2 \frac{\Theta^3}{3!}]$ and the inclusion is a (set-theoretical) equality.
Thus the special subvariety $V$ is isomorphic to the Brill-Noether locus $(V_2)'$ 
of a tetragonally-related covering.

Suppose now that the base locus of the linear system $|K_C \otimes H^*|$ does not
contain any points of $C_{\sing}$. Since $C$ is not a plane quintic, this
implies that $|K_C \otimes H^*|$ is base-point free.
By \cite[\S 2,Cor.]{Bea82b}, applied to the pull-back of the linear system to the normalised curve, we know that $S$ is irreducible if the linear system
$|K_C \otimes H^*|$ induces a map $\holom{\varphi}{C}{\PP^{p_a(C)-4}}$
that is birational onto its image $f(C)$.
Suppose now that this is not the case: then there exists for every generic point
$p \in C$ another generic point $q \in C$ such that
\[
h^0(C, K_C \otimes H^* \otimes  \sO_C(-p-q)) = h^0(C, K_C \otimes H^* \otimes  \sO_C(-p)).
\] 
By Riemann-Roch this implies that the linear system $|H \otimes  \sO_C(p+q)|$
is a base-point free $g^2_6$. Since $C$ is not hyperelliptic, we obtain in this way
a one-dimensional subset $W \subset W^2_6 C$.
Let $C'$ be the image of the morphism \holom{\varphi_{|H \otimes  \sO_C(p+q)|}}{C}{\PP^2}. 
Since $C$ is irreducible and not trigonal, the curve $C'$ is an irreducible cubic or sextic curve. 
Let \holom{\nu}{\upC}{C} be the normalisation, then we distinguish two cases:

{\em Case 1: $\upC$ is hyperelliptic.} 
Denote by $\holom{h}{\upC}{\PP^1}$ the hyperelliptic covering.
Since $H$ and $H \otimes  \sO_C(p+q)$ 
are base-point free, it is easy to see that $|\nu^*(H \otimes  \sO_C(p+q))|$ is a $g^3_6$.
Thus we have $\nu^*(H \otimes  \sO_C(p+q)) \simeq h^* \sO_{\PP^1}(3)$ and a factorisation
\[
\xymatrix{
T \ar[d]_h \ar[r]^{\nu} & C  \ar[d]^{\varphi_{|H \otimes  \sO_C(p+q)|}}
\\
\PP^1 \ar @{.>}[r]^{\bar \nu} & C' 
}
\]
In particular $\varphi_{|H \otimes  \sO_C(p+q)|}$ is not birational and $C'$ is a singular cubic.
Moreover if $x_1, x_2 \in \upC$ such that $\nu(x_1)=\nu(x_2)$
then $h(x_1)=h(x_2)$, unless $C'$ is nodal and $h(x_1)$ and $h(x_2)$ are mapped onto the unique node.

{\em Case 2: $\upC$ is not hyperelliptic.} 
In this case the pull-backs $\nu^*(H \otimes  \sO_C(p+q))$ define a one-dimensional subset
$\tilde W \subset W^2_6 \tilde C$. It follows by \cite[p.198]{ACGH85} that $\upC$ is bielliptic,
and if $\holom{h}{\upC}{E}$ is a two-to-one map onto an elliptic curve $E$, then
$\nu^*(H \otimes  \sO_C(p+q)) \simeq h^* L$ where $L \in \picd^3 E$.
As in the first case we have a factorisation $\holom{\bar \nu}{E}{C'}$ which is easily seen to
be an isomorphism. In particular $C$ is obtained from $\upC$ by identifying points
that are in a $h$-fibre.

\subsection{The irreducible components of $V^2$}
\label{subsectionz2}

Let $C'$ be a smooth curve of genus $g(C) \geq 6$ that is bielliptic, i.e.
we have a double cover \holom{p'}{C'}{E} onto an elliptic curve $E$.
As usual \holom{\pi'}{\upC'}{C'} will be an \'etale double cover, and we will
suppose that the morphism \holom{p' \circ \pi'}{\upC'}{E} is not 
Galois (in the terminology of \cite{Deb88, Nar92} the covering belongs to the 
family $\mathcal R'_{\mathcal B_{g(C)}} \subset \mathcal R_{g(C)}$,
cf. Remark \ref{remarkbielliptic}).

If we apply the tetragonal construction to a general $g^1_4$ on $C$, we
obtain a Beauville admissible cover \holom{\pi}{\upC}{C}
such that the normalisation $\holom{\nu}{T}{C}$ 
is a smooth hyperelliptic curve $T$ of genus $g(C)-2$.
Denote by \holom{h}{T}{\PP^1} the hyperelliptic structure.
Then $\nu$ identifies two pairs of points $x_1, x_2$
and $y_1, y_2$ such that $h(x_1), h(x_2), h(y_1), h(y_2)$ are four distinct points 
in $\PP^1$ (this follows from the \lq figure locale\rq \ in \cite[7.2.4]{Deb88}). 

By \cite[Ch.15]{Nar92} a tetragonal structure on $C$ can be constructed as follows:
there exists a unique double cover \holom{j}{\PP^1}{\PP^1}
sending each pair $h(x_1), h(x_2)$ and $h(y_1), h(y_2)$ onto a single point. 
The four-to-one covering \holom{j \circ h}{T}{\PP^1} factors through
the normalisation $\nu$, so we have a four-to-one cover
\holom{f}{C}{\PP^1}. 
If we apply the tetragonal construction to $H := f^* \sO_{\PP^1}(1)$,
we recover the original \'etale double cover \holom{\pi'}{\upC'}{C'}.
By Section \ref{subsectiontetragonal} we know that the
Brill-Noether locus $V^2$ associated to $\pi'$ is isomorphic to a special subvariety
associated to $|K_C \otimes H^*|$. 
By considering the exact sequence
\[
0 \rightarrow \nu_* (K_T \otimes \nu^* H^*) \rightarrow K_C \otimes H^* \rightarrow \C_{\nu(x_1)} \oplus \C_{\nu(y_1)}
\rightarrow 0
\]
one sees easily that the linear system $|K_C \otimes H^*|$ is base-point free 
(but does not separate the singular points $\nu(x_1)$ and $\nu(y_1)$).
Since the points $h(x_1), h(x_2), h(y_1), h(y_2)$ are distinct, it follows by the Case 1 in Section
\ref{subsectiontetragonal} that the special subvarieties are irreducible.
The following proposition summarises these considerations.

\begin{proposition} \label{propositionz2}
Let $C'$ be a smooth curve of genus $g(C) \geq 6$ that is bielliptic, i.e.
we have a double cover \holom{p'}{C'}{E} onto an elliptic curve $E$.
Let \holom{\pi'}{\upC'}{C'} be an \'etale double cover
such that the cover $\upC \rightarrow E$ is not Galois.
Then $V^2$ is irreducible.
\end{proposition}

\section{Proof of Theorem \ref{theoremirreduciblecomponents}}

If $V^2$ is reducible, we know by Corollary \ref{corollaryreducible}
that $C$ is trigonal, a plane quintic or bielliptic. The first two cases 
are settled in the Sections \ref{subsectionquintic} and \ref{subsectiontrigonal}.
If $C$ is bielliptic we distinguish two cases: the four-to-one cover $\upC \rightarrow C \rightarrow E$
is Galois or not. In the Galois case we conclude by Corollary \ref{corollarygone} 
and Proposition \ref{propositiongtwo}.
In the last case we use 
Proposition \ref{propositionz2}. q.e.d.


\end{document}